\documentclass[11pt]{amsart}
\usepackage{amssymb}
\usepackage{epsfig}
\usepackage{mathdots}
 \theoremstyle{plain}
\newtheorem{theorem}{Theorem}
\newtheorem{corollary}{Corollary}

\newtheorem{lemma}{Lemma}
\newtheorem{proposition}{Proposition}
\newtheorem{example}{Example}
\theoremstyle{definition}
\newtheorem{definition}{Definition}
\theoremstyle{remark}

\numberwithin{equation}{section}

\newcommand{\bT}{\begin{theorem}}
\newcommand{\eT}{\end{theorem}}
\newcommand{\bProp}{\begin{proposition}}
\newcommand{\eProp}{\end{proposition}}
\newcommand{\bE}{\begin{example}}
\newcommand{\eE}{\end{example}}
\newcommand{\bL}{\begin{lemma}}
\newcommand{\eL}{\end{lemma}}
\newcommand{\bP}{\begin{proof}}
\newcommand{\eP}{\end{proof}}
\newcommand{\bC}{\begin{corollary}}
\newcommand{\eC}{\end{corollary}}
\newcommand{\bD}{\begin{definition}}
\newcommand{\eD}{\end{definition}}
\newcommand{\be}{\begin{enumerate}}
\newcommand{\ee}{\end{enumerate}}
\newcommand{\beqa}{\begin{eqnarray*}}
\newcommand{\eeqa}{\end{eqnarray*}}
\newcommand{\beqaa}{\begin{eqnarray}}
\newcommand{\eeqaa}{\end{eqnarray}}
\newcommand{\ba}{\begin{array}}
\newcommand{\ea}{\end{array}}

\setbox0=\hbox{$+$}
\newdimen\plusheight
\plusheight=\ht0
\def\+{\;\lower\plusheight\hbox{$+$}\;}

\setbox0=\hbox{$-$}
\newdimen\minusheight
\minusheight=\ht0
\def\-{\;\lower\minusheight\hbox{$-$}\;}

\setbox0=\hbox{$\cdots$}
\newdimen\cdotsheight
\cdotsheight=\plusheight
\def\cds{\lower\cdotsheight\hbox{$\cdots$}}
\begin{document}

\title[Some Implications of the WP-Bailey Tree ]
       {Some Implications of the WP-Bailey Tree}
\author{James Mc Laughlin}
\address{Mathematics Department\\
 Anderson Hall\\
West Chester University, West Chester, PA 19383}
\email{jmclaughl@wcupa.edu}

\author{Peter Zimmer}
\address{Mathematics Department\\
 Anderson Hall\\
West Chester University, West Chester, PA 19383}
\email{pzimmer@wcupa.edu}


\date{\today}

\begin{abstract}
We consider a special case of a WP-Bailey chain of George Andrews,
and use it to derive a number of curious transformations of basic
hypergeometric series.

We also derive two new WP-Bailey pairs, and use them to derive some
additional new transformations for basic hypergeometric series.

Finally, we briefly consider the implications of WP-Bailey pairs\\
$(\alpha_n(a,k)$, $\beta_n(a,k))$, in which  $\alpha_n(a,k)$ is
independent of $k$, for generalizations of identities of the
Rogers-Ramanujan type.
\end{abstract}

\maketitle

\section{Introduction}

 Andrews, building on prior work of Bressoud
\cite{B81a}
 and Singh \cite{S94}, in \cite{A01}
 defined a \emph{WP-Bailey pair} to be  a pair of sequences
 $(\alpha_{n}(a,k),\,\beta_{n}(a,k))$ satisfying
 {\allowdisplaybreaks
\begin{align}\label{WPpair}
\beta_{n}(a,k) &= \sum_{j=0}^{n}
\frac{(k/a)_{n-j}(k)_{n+j}}{(q)_{n-j}(aq)_{n+j}}\alpha_{j}(a,k)\\
&= \frac{(k/a,k;q)_n}{(aq,q;q)_n}\sum_{j=0}^{n}
\frac{(q^{-n})_{j}(kq^n)_{j}}{(aq^{1-n}/k)_{j}(aq^{n+1})_{j}}
\left(\frac{qa}{k}\right)^j\alpha_{j}(a,k). \notag
\end{align}
}

Andrews also showed in \cite{A01} that there were two distinct ways
to construct new WP-Bailey pairs from a given pair. If
$(\alpha_{n}(a,k),\,\beta_{n}(a,k))$ satisfy \eqref{WPpair}, then so
do $(\alpha_{n}'(a,k),\,\beta_{n}'(a,k))$ and
$(\tilde{\alpha}_{n}(a,k),\,\tilde{\beta}_{n}(a,k))$, where
{\allowdisplaybreaks
\begin{align}\label{wpn1}
\alpha_{n}'(a,k)&=\frac{(\rho_1, \rho_2)_n}{(aq/\rho_1,
aq/\rho_2)_n}\left(\frac{k}{c}\right)^n\alpha_{n}(a,c),\\
\beta_{n}'(a,k)&=\frac{(k\rho_1/a,k\rho_2/a)_n}{(aq/\rho_1,
aq/\rho_2)_n} \notag\\
&\phantom{as}\times \sum_{j=0}^{n} \frac{(1-c
q^{2j})(\rho_1,\rho_2)_j(k/c)_{n-j}(k)_{n+j}}{(1-c)(k\rho_1/a,k\rho_2/a)_n(q)_{n-j}(qc)_{n+j}}
\left(\frac{k}{c}\right)^j\beta_{j}(a,c), \notag
\end{align}
}with  $c=k\rho_1 \rho_2/aq$ for the pair above, and
{\allowdisplaybreaks
\begin{align}\label{wpn2}
\tilde{\alpha}_{n}(a,k)&= \frac{(qa^2/k)_{2n}}{(k)_{2n}}\left
(\frac{k^2}{q a^2} \right)^n\alpha_{n} \left(a, \frac{q a^2}{k}
\right), \\
\tilde{\beta}_{n}(a,k)&=\sum_{j=0}^{n}
\frac{(k^2/qa^2)_{n-j}}{(q)_{n-j}}\left (\frac{k^2}{q a^2}
\right)^j\beta_{j} \left(a, \frac{q a^2}{k} \right). \notag
\end{align}
} These two constructions allow a ``tree" of WP-Bailey pairs to be
generated from a single WP-Bailey pair.  Andrews and Berkovich
\cite{AB02} further investigated these two branches of the WP-Bailey
tree, in the process deriving many new transformations for basic
hypergeometric series. Spiridonov \cite{S02} derived an elliptic
generalization of Andrews first WP-Bailey chain, and Warnaar
\cite{W03} \footnote{In a note added after submitting the paper
\cite{W03}, Warnaar remarks that he had discovered many more
transformations for basic and elliptic WP-Bailey pairs, and gives
two further examples of chains that hold at the elliptic level. }
added four new branches to the WP-Bailey tree, two of which had
generalizations to the elliptic level. More recently, Liu and Ma
\cite{LM08} introduced the idea of a general WP-Bailey chain (as a
solution to a system of linear equations), and added one new branch
to the WP-Bailey tree.

In the present paper, we derive two new WP-Bailey pairs, one of them
restricted in the sense that it is necessary to set $k=q$. We then
insert these in some of the WP-Bailey chains listed above to derive
new transformations of basic hypergeometric series.

We also consider a special case ($k=aq$) of the first WP-Bailey
chain of Andrews, and show how it leads some unusual transformations
of series.

We also briefly consider the special case of a WP-Bailey pair
$(\alpha_n(a),$ $\beta_n(a,k))$, where the $\alpha_n$ are
independent of $k$. We show how a such pair may give rise to a
generalization of a Slater-type identity deriving from the standard
Bailey pair $(\alpha_n(a), \beta_n(a,0))$.

\section{Finite Basic Hypergeometric Identities deriving from Simple WP-Bailey pairs}

In this section we derive some unusual transformations from
\eqref{wpn1} by inserting ``simple" WP-Bailey pairs (see below for
the definition).

We reformulate the constructions at \eqref{wpn1} and (for later use)
\eqref{wpn2} as
 transformations relating the original WP-Bailey pair
$(\alpha_{n}(a,k),\,\beta_{n}(a,k))$. We first recall the following
elementary transformation:
\begin{equation}\label{aqn-r}
(a;q)_{n-r} = \frac{ q^{r(r+1)/2}(a;q)_n}
{(q^{1-n}/a;q)_r(-aq^{n})^r}.
\end{equation}

\begin{theorem}\label{wpbt}
If $(\alpha_n(a,k),\beta_n(a,k))$ satisfy
\begin{equation}\label{wpbtabcons}
\beta_{n}(a,k) = \sum_{j=0}^{n}
\frac{(k/a)_{n-j}(k)_{n+j}}{(q)_{n-j}(aq)_{n+j}}\alpha_{j}(a,k),
\end{equation}
then {\allowdisplaybreaks
\begin{multline}\label{wpbteq1}
\sum_{n=0}^{N} \frac{(1-kq^{2n})(\rho_1,\rho_2,k a
q^{N+1}/\rho_1\rho_2,q^{-N};q)_n }{(1-k)(k q/\rho_1,k
q/\rho_2,\rho_1 \rho_2 q^{-N}/a,k q^{1+N};q)_n}\,q^n \beta_n(a,k)\\=
\frac{(k q,k q/\rho_1 \rho_2,a q/\rho_1,a q/\rho_2;q)_{N}} {(k
q/\rho_1,k q/\rho_2,a q/\rho_1 \rho_2,a
q;q)_{N}}\phantom{asdadasdasdabvvbvmvmbnvbnvdassdas}\\
\times \sum_{n=0}^{N} \frac{(\rho_1,\rho_2, k a q^{N+1}/\rho_1
\rho_2,q^{-N};q)_{n}} {(a q/\rho_1,a q/\rho_2,a q^{1+N}, \rho_1
\rho_2 q^{-N}/k;q)_{n}}\left ( \frac{a q}{k}\right)^n \alpha_n(a,k),
\end{multline}
}and {\allowdisplaybreaks \begin{multline}\label{wpbteq2}
\sum_{n=0}^{N} \frac{(q^{-N};q)_n }{(q^{-N}k^2/a^2;q)_n}\,q^n
\beta_n(a,k)= \frac{(qa/k,qa^2/k;q)_{N}}
{(q a,qa^2/k^2;q)_{N}}\\
\times \sum_{n=0}^{N} \frac{( q^{N+1}a^2/k,q^{-N};q)_{n}(k;q)_{2n}}
{(a q^{1+N},  q^{-N}k/a;q)_{n}(q a^2/k;q)_{2n}}\left ( \frac{a
q}{k}\right)^n \alpha_n(a,k).
\end{multline}
}
\end{theorem}
\begin{proof}
The identity at \eqref{wpbteq1} follows from \eqref{wpn1}, after
substituting for $\alpha'_{n}(a,k)$ in \eqref{WPpair}, employing
\eqref{aqn-r}, then setting the two expressions for
$\beta'_{n}(a,k)$ equal, and finally replacing $c$ with $k$. The
identity at \eqref{wpbteq2} follows from \eqref{wpn2}, after setting
$k=q a^2/c$, using \eqref{aqn-r}, and finally replacing $c$ with
$k$.
\end{proof}

For later use we also note the following corollary, which is
immediate upon letting $N\to \infty$.
\begin{corollary}
If $(\alpha_n(a,k),\beta_n(a,k))$ satisfy
\begin{equation*} \beta_{n}(a,k) = \sum_{j=0}^{n}
\frac{(k/a)_{n-j}(k)_{n+j}}{(q)_{n-j}(aq)_{n+j}}\alpha_{j}(a,k),
\end{equation*}
then
\begin{multline}\label{wpbteq1b}
\sum_{n=0}^{\infty} \frac{(1-kq^{2n})(\rho_1,\rho_2;q)_n }{(1-k)(k
q/\rho_1,k q/\rho_2;q)_n}\,\left(\frac{a q}{\rho_1 \rho_2}\right)^n
\beta_n(a,k)=\\\frac{(k q,k q/\rho_1\rho_2,a q/\rho_1,a
q/\rho_2;q)_{\infty}}{(k q/\rho_1,k q/\rho_2,a q/\rho_1 \rho_2,a
q;q)_{\infty}}
\sum_{n=0}^{\infty} \frac{(\rho_1,\rho_2;q)_{n}} {(a q/\rho_1,a
q/\rho_2;q)_{n}}\left ( \frac{a q}{\rho_1 \rho_2}\right)^n
\alpha_n(a,k),
\end{multline}
and
\begin{multline}\label{wpbteq2b}
\sum_{n=0}^{\infty}  \left(\frac{q a^2}{k^2}\right)^n \beta_n(a,k)\\
= \frac{(qa/k,qa^2/k;q)_{\infty}} {(q a,qa^2/k^2;q)_{\infty}}
\sum_{n=0}^{\infty} \frac{(k;q)_{2n}} {(q a^2/k;q)_{2n}}\left (
\frac{q a^2}{k^2}\right)^n \alpha_n(a,k).
\end{multline}
\end{corollary}
Remark: At several places throughout the paper we replace $a$ with
$x/q$, $\rho_1$ with $y$ and  $\rho_2$ with $z$, in order to more
easily make comparisons with a transformation due to Bailey
\eqref{Baileyeq} later.

We now consider some applications of the following corollary.

\begin{corollary}\label{f6c1}
Let $N$ be a positive integer. Suppose the sequences $\{\alpha_n\}$
and $\{ \beta_n\}$ are related by
{\allowdisplaybreaks\begin{align}\label{abcons1}
 \beta_n
&=\sum_{r=0}^{n}\alpha_{r}.
\end{align}
} Then
\begin{multline}\label{abcons1eq}
\sum_{n=0}^{N}
\frac{(q\sqrt{k},-q\sqrt{k},y,z,k^2q^{N}/yz,q^{-N};q)_n}
{(\sqrt{k},-\sqrt{k},qk/y,qk/z,yzq^{1-N}/k,kq^{1+N};q)_n}\,q^n\,\beta_n\\
= \frac{(1-k/y)(1-k/z)(1-kq^N)(1-kq^N/yz)}
{(1-k)(1-k/yz)(1-kq^N/y)(1-kq^N/z)}\\
\times \sum_{n=0}^{N} \frac{(y,z, k^2q^{N}/yz,q^{-N};q)_{n}}
{(k/y,k/z,kq^{N}, yzq^{-N}/k;q)_{n}} \alpha_n.
\end{multline}
\end{corollary}
\begin{proof} Let $\rho_1 = y$, $\rho_2=z$ and $a=k/q$  in
\eqref{wpbtabcons} and \eqref{wpbteq1} in  Theorem \ref{wpbt}.
\end{proof}

 It is interesting
that the simple condition relating $\alpha_n$ and $\beta_n$ at
\eqref{abcons1} (we might call such pairs ``simple" Bailey pairs)
should have several non-trivial consequences. We had previously
discovered (in \cite{MZ07}) the result that follows from Corollary
\ref{f6c1} upon letting $N \to \infty$ (found by another method),
so in fact many of the results in \cite{MZ07} has a finite
counterpart. We give some examples.

\begin{corollary}\label{f6c1-1}
Let $N$ be a positive integer.  Then
\begin{multline}\label{c1-1eq}
\sum_{n=0}^{\lfloor N/2 \rfloor}
\frac{(q\sqrt{k},-q\sqrt{k},y,z,k^2q^{N}/yz,q^{-N};q)_{2n}q^{2n}}
{(\sqrt{k},-\sqrt{k},qk/y,qk/z,yzq^{1-N}/k,kq^{1+N};q)_{2n}}\\
= \frac{(1-k/y)(1-k/z)(1-kq^N)(1-kq^N/yz)}
{(1-k)(1-k/yz)(1-kq^N/y)(1-kq^N/z)}\\
\times \sum_{n=0}^{N} \frac{(y,z, k^2q^{N}/yz,q^{-N};q)_{n}}
{(k/y,k/z,kq^{N}, yzq^{-N}/k;q)_{n}} (-1)^n;
\end{multline}
\begin{multline}\label{c1-1eq1}
\sum_{n=0}^{\lfloor N/2 \rfloor}
\frac{(1-yq^{4n+1})(y,-\sqrt{y}q^{N+1},q^{-N};q)_{2n}q^{2n}}
{(1-y q)(q^2,-\sqrt{y}q^{1-N},yq^{2+N};q)_{2n}}\\
= \frac{(1-q)(1+\sqrt{y})}
{(1-q^{N+1})(1+\sqrt{y}q^N)}\frac{(y q^2,-q;q)_N}{(q\sqrt{y},-1/\sqrt{y};q)_N}\,y^{-N/2}.
\end{multline}
\end{corollary}
\begin{proof} Let $\alpha_n=(-1)^n$ in Corollary \ref{f6c1} to get \eqref{c1-1eq}. The identity at \eqref{c1-1eq1} follows from \eqref{c1-1eq}, after setting $k=qy$ and $z=-q\sqrt{y}$, using the finite form of the $q$-Dixon sum ((II.14) on page 355 of \cite{GR04}) to sum the series on the right side, and finally using the identity at (I.9) on page 351 of \cite{GR04} to rearrange two of the $q$-products in this sum.
\end{proof}
Remark: The left side in \eqref{c1-1eq} above could be expressed
 as
\[
_{12}\,W_{11}
\left(k;y,yq,z,zq,\frac{k^2q^{N}}{yz},\frac{k^2q^{1+N}}{yz},q^{-N},q^{1-N},q^2;q^2;q^2\right),
\]
where the additional $q$-products are inserted here to give the
series the form of a $_{r+1}\,W_{r}$ series. The identity at
\eqref{c1-1eq}, which can thus be regarded as a transformation
between a $_{12}\,\phi_{11}$ series with base $q^2$ and a
$_5\phi_4$ series with base $q$, does not appear to be a special
case of other similar transformations due to Bailey (which involve
$_{12}\,\phi_{11}$ series with base $q$ - see \cite{GR04}, pages
46--47) or Jain and Verma \cite{JV80} (which involve
$_{12}\,\phi_{11}$ series with base $q$ or $q^3$).

Recall that a $\,_{r+1}W_{r}$  series is defined by
\[
_{r+1} W_{r}(a_1;a_4,\dots a_{r+1};q,z):=\,_{r+1}\phi_{r} \left [
\begin{matrix}
a_{1},\,q\sqrt{a_1},\,-q\sqrt{a_1},\,a_4,\,\dots, \, a_{r+1}\\
\sqrt{a_1},\,-\sqrt{a_1},\, \frac{a_1 q}{a_4},\,\dots , \frac{a_1
q}{a_{r+1}}
\end{matrix}
; q, z \right ],
\]
where, as usual, an $_{r+1} \phi _{r}$ basic hypergeometric series
is defined by {\allowdisplaybreaks
\begin{equation*} _{r+1} \phi _{r} \left [
\begin{matrix}
a_{1}, a_{2}, \dots, a_{r+1}\\
b_{1}, \dots, b_{r}
\end{matrix}
; q,x \right ] =  \sum_{n=0}^{\infty}
\frac{(a_{1};q)_{n}(a_{2};q)_{n}\dots (a_{r+1};q)_{n}}
{(q;q)_{n}(b_{1};q)_{n}\dots (b_{r};q)_{n}} x^{n}.
\end{equation*}
}

\begin{corollary}\label{f6c11}
Let $N$ be a positive integer.  Then {\allowdisplaybreaks
\begin{multline}\label{cor4eq1}
\sum_{n=0}^{N}
\frac{(q\sqrt{k},-q\sqrt{k},y,z,k^2q^{N}/yz,q^{-N};q)_n  (n+1)\, q^n
}
{(\sqrt{k},-\sqrt{k},qk/y,qk/z,yzq^{1-N}/k,kq^{1+N};q)_n}\\
= \frac{(1-k/y)(1-k/z)(1-kq^N)(1-kq^N/yz)}
{(1-k)(1-k/yz)(1-kq^N/y)(1-kq^N/z)}\\
\times \sum_{n=0}^{N} \frac{(y,z, k^2q^{N}/yz,q^{-N};q)_{n}}
{(k/y,k/z,kq^{N}, yzq^{-N}/k;q)_{n}}.
\end{multline}
}
\end{corollary}
\begin{proof} Let $\alpha_n=1$ in Corollary \ref{f6c1}.
\end{proof}

\begin{corollary}\label{f6c4}
\begin{multline}\label{f8phi71}
_{8} \phi _{7} \left [
\begin{matrix}
q\sqrt{k},-q\sqrt{k}, y,z, a q,b q,k^2q^{N}/yz,q^{-N}\\
\sqrt{k},-\sqrt{k}, qk/y,qk/z,(a+b-1)q,yzq^{1-N}/k,kq^{1+N}
\end{matrix}
; q,q \right ]\\
= \frac{(1-k/y)(1-k/z)(1-kq^N)(1-kq^N/yz)}
{(1-k)(1-k/yz)(1-kq^N/y)(1-kq^N/z)}\\ \times \;_{6}\phi_{5} \left
[\begin{matrix}
y, z, a, b,k^2q^{N}/yz,q^{-N}\\
k/y,k/z, (a+b-1)q,kq^{N}, yzq^{-N}/k
\end{matrix}
; q,q^2\right ].
\end{multline}
\end{corollary}

\begin{proof}
In Corollary \ref{f6c1}, define $\alpha_0 =1$, and for $n>0$,
\begin{equation*}
 \alpha_n =
 \frac{(aq,bq;q)_n}{((a+b-1)q,q;q)_n}-\frac{(aq,bq;q)_{n-1}}{((a+b-1)q,q;q)_{n-1}}
 =\frac{(a,b;q)_n  q^{2n}}{((a+b-1)q,q;q)_n}.
 \end{equation*}
 The sum for $\beta_n$ is easily seen to telescope to give
 \[
 \beta_n =\frac{(aq,bq;q)_n}{((a+b-1)q,q;q)_n}.
 \]
and the result follows.
\end{proof}

\begin{corollary}\label{f6c5}
{\allowdisplaybreaks
\begin{multline}\label{f8phi72}
_{8} \phi _{7} \left [
\begin{matrix}
q\sqrt{k},-q\sqrt{k}, y,z, a q,b q,k^2q^{N}/yz,q^{-N}\\
\sqrt{k},-\sqrt{k}, qk/y,qk/z,a b q,yzq^{1-N}/k,kq^{1+N}
\end{matrix}
; q,q \right ]\\
= \frac{(1-k/y)(1-k/z)(1-kq^N)(1-kq^N/yz)}
{(1-k)(1-k/yz)(1-kq^N/y)(1-kq^N/z)}\\ \times \;_{6}\phi_{5} \left
[\begin{matrix}
y, z, a, b,k^2q^{N}/yz,q^{-N}\\
k/y,k/z, a b q,kq^{N}, yzq^{-N}/k
\end{matrix}
; q,q\right ].
\end{multline}
}
\end{corollary}

\begin{proof}
This time, in Corollary \ref{f6c1}, define $\alpha_0 =1$, and for
$n>0$,
\begin{equation*}
 \alpha_n =
 \frac{(aq,bq;q)_n}{(a b q,q;q)_n}-\frac{(aq,bq;q)_{n-1}}{(a b q,q;q)_{n-1}}
 =\frac{(a,b;q)_n  q^{n}}{(a bq,q;q)_n}.
 \end{equation*}
 The  result follows as above.
\end{proof}

\begin{corollary}\label{f6c111}
Let $N$ and $m$ be  positive integers and let $p$ be an integer.
Then {\allowdisplaybreaks
\begin{multline}\label{cor7eq1}
\sum_{n=0}^{N}
\frac{(q\sqrt{k},-q\sqrt{k},y,z,k^2q^{N}/yz,q^{-N};q)_nq^{(mn^2+(p+2)n)/2}
}
{(\sqrt{k},-\sqrt{k},qk/y,qk/z,yzq^{1-N}/k,kq^{1+N};q)_n(-q^{(m+p)/2};q^{m})_n}\\
= \frac{(1-k/y)(1-k/z)(1-kq^N)(1-kq^N/yz)}
{(1-k)(1-k/yz)(1-kq^N/y)(1-kq^N/z)}\\
\times \left(1-q^{(m-p)/2}\sum_{n=1}^{N} \frac{(y,z,
k^2q^{N}/yz,q^{-N};q)_{n}q^{(mn^2+(p-2m)n)/2}} {(k/y,k/z,kq^{N},
yzq^{-N}/k;q)_{n}(-q^{(m+p)/2};q^{m})_n} \right).
\end{multline}
}
\end{corollary}
\begin{proof} In Corollary \ref{f6c1} set $\alpha_0=1$ and, for $n>0$,
\[
\alpha_{n}=\frac{q^{(m n^2+p
n)/2}}{(-q^{(m+p)/2};q^m)_n}-\frac{q^{(m (n-1)^2+p
(n-1))/2}}{(-q^{(m+p)/2};q^m)_{n-1}} =-q^{(m-p)/2}\frac{q^{(m
n^2+(p-2m) n)/2}}{(-q^{(m+p)/2};q^m)_n}.
\]
\end{proof}

\begin{corollary}\label{pb2}
Let $P$, $p$, $Q$, $q$, $R$, $a$, $b$, $c$, $k$, $y$ and $z$ be
complex numbers such that none of the denominators below vanish.
Then {\allowdisplaybreaks
\begin{multline}\label{poly2} \sum_{n=0}^{N}
\frac{\left(q\sqrt{k},-q\sqrt{k},y,z,\frac{k^2q^{N}}{yz},q^{-N};q\right)_n}
{\left(\sqrt{k},-\sqrt{k},\frac{qk}{y},\frac{qk}{z},
\frac{yzq^{1-N}}{k},kq^{1+N};q\right)_n}\\
\times \frac{ (ap^2;p^2)_n \left(b P^2;P^2 \right)_n  (cR^2;R^2)_n
\left(\frac{a Q^2}{b c};Q^2 \right)_n  \,q^{n}} {
\left(\frac{PQR}{p};\frac{PQR}{p}\right)_n \left(\frac{a p P Q}{c
R};\frac{p P Q}{R}\right)_n  \left(\frac{a
pQR}{bP};\frac{pQR}{P}\right)_n \left(\frac{b c p P R}{Q};\frac{ p P
R}{Q}\right)_n}
\\
=
\frac{\left(1-\frac{k}{y}\right)\left(1-\frac{k}{z}\right)(1-kq^N)\left(1-\frac{kq^N}{yz}\right)}
{(1-k)\left(1-\frac{k}{yz}\right)\left(1-\frac{kq^N}{y}\right)\left(1-\frac{kq^N}{z}\right)}
\sum_{n=0}^{N} \frac{\left(y,z,
\frac{k^2q^{N}}{yz},q^{-N};q\right)_{n}}
{\left(\frac{k}{y},\frac{k}{z},kq^{N}, \frac{yzq^{-N}}{k};q\right)_{n}}\\
\times \frac{ \left(1-a p^n P^n Q^n R^n\right)\left(1-b \frac{p^n
P^n}{ Q^{n} R^{n}}\right) \left(1-\frac{ P^n Q^n
  }{c p^{n} R^{n}}\right) \left(1-\frac{a p^n
   Q^n }{b cP^{n}R^{n}}\right)}{(1-a) (1-b)
   \left(1-\frac{1}{c}\right) \left(1-\frac{a}{b c}\right)}
\\
\times
 \frac{ (a;p^2)_n
\left(b ;P^2 \right)_n  (c;R^2)_n \left(\frac{a }{b c};Q^2
\right)_n\, R^{2n} } { \left(\frac{PQR}{p};\frac{PQR}{p}\right)_n
\left(\frac{a p P Q}{c R};\frac{p P Q}{R}\right)_n \left(\frac{a
pQR}{bP};\frac{pQR}{P}\right)_n \left(\frac{b c p P R}{Q};\frac{ p P
R}{Q}\right)_n};
\end{multline}
} {\allowdisplaybreaks\begin{multline}\label{poly2q} \sum_{n=0}^{N}
\frac{\left(q\sqrt{k},-q\sqrt{k},y,z,\frac{k^2q^{N}}{yz},q^{-N};q\right)_n}
{\left(\sqrt{k},-\sqrt{k},\frac{qk}{y},\frac{qk}{z},\frac{yzq^{1-N}}{k},kq^{1+N};q\right)_n}
\frac{ (aq^m,bq^m,cq^m,\frac{a q^m}{b c};q^m)_n } {
\left(\frac{a}{c}q^m,\frac{a}{b}q^m,b cq^m,q^m;q^m\right)_n }\,q^{n}
\\
=
\frac{\left(1-\frac{k}{y}\right)\left(1-\frac{k}{z}\right)(1-kq^N)\left(1-\frac{kq^N}{yz}\right)}
{(1-k)\left(1-\frac{k}{yz}\right)\left(1-\frac{kq^N}{y}\right)\left(1-\frac{kq^N}{z}\right)}
\times\\ \sum_{n=0}^{N} \frac{\left(y,z,
\frac{k^2q^{N}}{yz},q^{-N};q\right)_{n}}
{\left(\frac{k}{y},\frac{k}{z},kq^{N},
\frac{yzq^{-N}}{k};q\right)_{n}}\frac{
(q^m\sqrt{a},-q^m\sqrt{a},a,b,c,\frac{a }{b c};q^m)_n \, q^{mn}} {
\left(\sqrt{a},-\sqrt{a},\frac{a}{c}q^m,\frac{a}{b}q^m,b
cq^m,q^m;q^m\right)_n }.
\end{multline}}
\end{corollary}

\begin{proof}
We use the special case $m=0$, $d=1$ of the identity of Subbarao and
Verma labeled (2.2) in \cite{SV99}, namely,
{\allowdisplaybreaks\begin{multline} \sum_{k=0}^{n}  \frac{
\left(1-a p^k P^k Q^k R^k\right)\left(1-b \frac{p^k P^k}{ Q^{k}
R^{k}}\right) \left(1-\frac{ P^k Q^k
  }{c p^{k} R^{k}}\right) \left(1-\frac{a p^k
   Q^k }{b cP^{k}R^{k}}\right)}{(1-a) (1-b)
   \left(1-\frac{1}{c}\right) \left(1-\frac{a}{b c}\right)}
\\
\times
 \frac{ (a;p^2)_k
\left(b ;P^2 \right)_k  } {
\left(\frac{PQR}{p};\frac{PQR}{p}\right)_k \left(\frac{a p P Q}{c
R};\frac{p P Q}{R}\right)_k}
 \frac{ (c;R^2)_k
\left(\frac{a }{b c};Q^2 \right)_k } { \left(\frac{a
pQR}{bP};\frac{pQR}{P}\right)_k \left(\frac{b c p P R}{Q};\frac{ p P
R}{Q}\right)_k}\, R^{2k}\\
= \frac{ (ap^2;p^2)_n \left(b P^2;P^2 \right)_n (cR^2;R^2)_n
\left(\frac{a Q^2}{b c};Q^2 \right)_n  } {
\left(\frac{PQR}{p};\frac{PQR}{p}\right)_n \left(\frac{a p P Q}{c
R};\frac{p P Q}{R}\right)_n\left(\frac{a
pQR}{bP};\frac{pQR}{P}\right)_n \left(\frac{b c p P R}{Q};\frac{ p P
R}{Q}\right)_n},
\end{multline}}
and then, in \eqref{abcons1eq} above, let $\alpha_i$ be the $i$-th
term in the sum above, and let $\beta_n$ be the quantity on the
right side above.

The identity at \eqref{poly2q} follows upon setting
$P=Q=p=R=q^{m/2}$ and simplifying.
\end{proof}

Apart from the first chain of Andrews, there is only one other
WP-Bailey chain, of those seven alluded to in the introduction, that
leads to a non-trivial transformation similar to that in Corollary
\ref{f6c1} (the consequences of letting $k=aq$ in the other chains
being entirely trivial). This is the third chain of Warnaar (Theorem
2.5 in \cite{W03}). This chain implies that if
$(\alpha_{n}(a,k),\,\beta_{n}(a,k))$ satisfy \eqref{WPpair}, then
{\allowdisplaybreaks
\begin{multline}\label{War3}
\sum_{j=0}^n\frac{1+a q^{2j}}{1+a}\frac{(m/a;q^2)_{n-j}(a
m;q^2)_{n+j}}{(q^2;q^2)_{n-j}(
a^2q^2;q^2)_{n+j}}q^{-j}\alpha_j(a,m)\\
=q^{-n}\frac{(-m q;q)_{2n}}{(-a;q)_{2n}}\\ \times \sum_{j=0}^n
\frac{1-m q^{2j}}{1-m}\frac{(a/m;q^2)_{n-j}(a
m;q^2)_{n+j}}{(q^2;q^2)_{n-j}(m^2q^2;q^2)_{n+j}}\left(
\frac{m}{a}\right)^{n-j}\beta_j(a,m).
\end{multline}
} Upon setting $m=a q$, we get that if
$\beta_n=\sum_{j=0}^n\alpha_j$, then
\begin{multline}\label{War3triv}
\sum_{j=0}^n\frac{1+a q^{2j}}{1+a}\frac{(q;q^2)_{n-j}(a^2
q;q^2)_{n+j}}{(q^2;q^2)_{n-j}(
a^2q^2;q^2)_{n+j}}q^{-j}\alpha_j\\
=\frac{(1+a q^{2n})(1+a q^{2n+1})}{(1+a)(1+a q)}
\sum_{j=0}^n\frac{1-aq^{2j+1}}{1-a q}\frac{(1/q;q^2)_{n-j}(a^2
q;q^2)_{n+j}}{(q^2;q^2)_{n-j}(a^2q^4;q^2)_{n+j}} q^{-j}\beta_j.
\end{multline}

However, we do not pursue the consequences of this transformation
further here.



\section{New WP-Bailey Pairs}
We next exhibit two new WP-Bailey pairs.

\begin{lemma}\label{l1}
The pair $(\alpha_n^{(1)}(a,k),\beta_n^{(1)}(a,k))$ is a WP-Bailey
pair, where
\begin{align}\label{mz01}
\alpha_n^{(1)}(a,k)&=\frac{(q a^2/k^2;q)_n}{(q;q)_n}\left(
\frac{k}{a}\right)^n,\\
\beta_n^{(1)}(a,k)&=\frac{(q
a/k,k;q)_n}{(k^2/a,q;q)_n}\frac{(k^2/a;q)_{2n}}{(a q;q)_{2n}}.\notag
\end{align}
\end{lemma}
\begin{proof}
{\allowdisplaybreaks
\begin{align*}
\sum_{j=0}^{n}
\frac{(k/a)_{n-j}(k)_{n+j}}{(q)_{n-j}(aq)_{n+j}}&\alpha_j^{(1)}(a,k)\\
&= \frac{(k/a,k;q)_n}{(aq,q;q)_n}\sum_{j=0}^{n}
\frac{(q^{-n},kq^n;q)_{j}}{(aq^{1-n}/k,aq^{n+1};q)_{j}}
\left(\frac{qa}{k}\right)^j\alpha_j^{(1)}(a,k)\\
&= \frac{(k/a,k;q)_n}{(aq,q;q)_n}\sum_{j=0}^{n} \frac{(q^{-n},kq^n,q
a^2/k^2;q)_{j}}{(aq^{1-n}/k,aq^{n+1},q;q)_{j}}\,q^j\\
&= \frac{(k/a,k;q)_n}{(aq,q;q)_n}\frac{(a q/k,k^2q^n/a;q)_n}{(aq^{n+1},k/a;q)_n}\\
&=\frac{(q a/k,k;q)_n}{(k^2/a,q,q)_n}\frac{(k^2/a;q)_{2n}}{(a
q,q)_{2n}}=\beta_n^{(1)}(a,k).
\end{align*}
} The third equality follows from the $q$-Pfaff-Saalsch\"{u}tz
sum,
\begin{equation}\label{qpfsaal}
 _{3} \phi _{2} \left [
\begin{matrix}
a, b,  q^{-n}\\
c, a b q^{1-n}/c
\end{matrix}
; q,q \right ] = \frac{(c/a,c/b;q)_n}{(c,c/ab;q)_n}.
\end{equation}
\end{proof}

\begin{lemma}\label{l3}
The pair $(\alpha_n^{(2)}(a,q),\beta_n^{(2)}(a,q))$ is a WP-Bailey
pair, where {\allowdisplaybreaks
\begin{align}\label{mz03}
\alpha_n^{(2)}(a,q)&=\frac{(a, \,q \sqrt{a},\, -q \sqrt{a}, \,d,\,
q/d, \, -a;q)_n} {(\sqrt{a},\,-\sqrt{a},\, aq/d,\,a d,\,
-q,\,q;q)_n}\,(-1)^n,\\
\beta_n^{(2)}(a,q)&=\begin{cases}
\displaystyle{\frac{(q^2/ad,dq/a;q^2)_{n/2}}{(adq,aq^2/d;q^2)_{n/2}}}, & n \text{ even},\\
\displaystyle{-a\frac{(q/ad,d/a;q^2)_{(n+1)/2}}{(ad,aq/d;q^2)_{(n+1)/2}}},&
n \text{ odd}.
\end{cases}\notag
\end{align}
}
\end{lemma}
Remark: Note that this pair is restricted in that it is necessary to
set $k=q$ for \eqref{WPpair} to hold.
\begin{proof}
{\allowdisplaybreaks
\begin{align*}
&\sum_{j=0}^{n}
\frac{(q/a)_{n-j}(q)_{n+j}}{(q)_{n-j}(aq)_{n+j}}\alpha_j^{(2)}(a,q)\\
&= \frac{(q/a;q)_n}{(aq;q)_n}\sum_{j=0}^{n}
\frac{(q^{-n},q^{n+1};q)_{j}}{(aq^{-n},aq^{n+1};q)_{j}}
a^j\alpha_j^{(2)}(a,q)\\
&= \frac{(q/a;q)_n}{(aq;q)_n}\sum_{j=0}^{n}
\frac{(q^{-n},q^{n+1},a,q \sqrt{a},-q
\sqrt{a},d,\frac{q}{d},-a;q)_{j}}
{(aq^{-n},aq^{n+1},\sqrt{a},-\sqrt{a},\frac{q
a}{d},a d,-q,q;q)_{j}}\left(-a\right)^j\\
&= \frac{(q/a;q)_n}{(aq;q)_n}\begin{cases} -a\displaystyle{\frac{(a
q;q)_{n}}{(q/a;q)_{n}}\frac{(q/ad,d/a;q^2)_{(n+1)/2}}{( a d,
qa/d;q^2)_{(n+1)/2}}},& n \text{
odd} \\
\displaystyle{\frac{(a
q;q)_{n}}{(q/a;q)_{n}}\frac{(q^2/ad,dq/a;q^2)_{n/2}}{( a dq,
q^2a/d;q^2)_{n/2}}},& n \text{ even},
\end{cases}\\
&=\beta_n^{(2)}(a,q).
\end{align*}
} The third equality follows from a $q$-analogue of Whipple's $_3
F_2$ sum,
\begin{multline}
\label{whq3f2} _{8} \phi _{7} \left [
\begin{matrix}
C, \,q \sqrt{C},\, -q \sqrt{C}, \,a,\, q/a,\, -C,\,d, \,
q/d\\
\sqrt{C},\,-\sqrt{C},\, C q/a,\,a C,\,-q,\,C q/d,\, C d
\end{matrix}
; q,-C  \right ]\\ = \frac{(C,C q;q)_{\infty} (a C d, a C q/d, C d
q/a,C q^2/a d;q^2)_{\infty}}{( C d, C q/d, a C, C q/a;q)_{\infty}},
\end{multline}
and upon setting $a=q^{-n}$, {\allowdisplaybreaks
\begin{multline*}
_{8} \phi _{7} \left [
\begin{matrix}
C, \,q \sqrt{C},\, -q \sqrt{C}, \,d,\, q/d,\, -C,\,q^{-n}, \,
q^{1+n}\\
\sqrt{C},\,-\sqrt{C},\, C q/d,\,d C,\,-q,\,C q^{1+n},\, C q^{-n}
\end{matrix}
; q,-C  \right ]\\ = \begin{cases}\displaystyle{\frac{( C
q;q)_{n}}{(q/C;q)_{n}}\frac{(q^2/dC,dq/C;q^2)_{n/2}}{( d C q, C
q^2/d;q^2)_{n/2}}} ,& n \text{
even} \\
\displaystyle{\frac{( C
q;q)_{n}}{(q/C;q)_{n}}\frac{(q/dC,d/C;q^2)_{(n+1)/2}}{( d C , C
q/d;q^2)_{(n+1)/2}}}(-C),& n \text{ odd}.
\end{cases}
\end{multline*}
}
\end{proof}

We had initially thought that the following WP-Bailey pair was
totally new also.

\begin{align}\label{mz02}
\alpha_n^{(3)}(a,k)&=\frac{(a, \,q \sqrt{a},\, -q \sqrt{a},
\,k/a,\,a \sqrt{q/k}, \, -a \sqrt{q/k};q)_n}
{(\sqrt{a},\,-\sqrt{a},\, q
a^2/k,\,\sqrt{qk},\, -\sqrt{qk},\,q;q)_n}\,(-1)^n,\\
\beta_n^{(3)}(a,k)&=\begin{cases}
\displaystyle{\frac{(k,k^2/a^2;q^2)_{n/2}}{(q^2, q^2
a^2/k;q^2)_{n/2}}}, & n \text{ even},\\
0,& n \text{ odd}.
\end{cases}\notag
\end{align}

However, this pair can be derived from an existing WP-Bailey pair
using a result of Warnaar \cite{W03}.
\begin{lemma}[Warnaar, \cite{W03}]\label{Wlem}
For $a$ and $k$ indeterminates the following equations are
equivalent: {\allowdisplaybreaks
\begin{align}\label{WPequiv}
\beta_{n}(a,k;q) &= \sum_{j=0}^{n}
\frac{(k/a)_{n-j}(k)_{n+j}}{(q)_{n-j}(aq)_{n+j}}\alpha_{j}(a,k;q),\\
\alpha_{n}(a,k;q) &= \frac{1-a q^{2n}}{{1-a}}\sum_{j=0}^{n}\frac{1-k
q^{2n}}{{1-k}}
\frac{(a/k)_{n-j}(a)_{n+j}}{(q)_{n-j}(kq)_{n+j}}\left(\frac{k}{a}
\right)^{n-j}\beta_{j}(a,k;q). \notag
\end{align}
}
\end{lemma}
If $a$ and $k$ are interchanged in the second equation, we see that
Lemma \ref{Wlem} implies the following.
\begin{corollary}\label{warcor}
If $(\alpha_{n}(a,k), \beta_n(a,k))$ are a WP-Bailey pair, then so
are\\ $(\alpha_{n}'(a,k), \beta_n'(a,k))$, where
\begin{align*}
\alpha_{n}'(a,k)&=\frac{1-a q^{2n}}{1-a}\left (\frac{k}{a} \right)^n
\beta_{n}(k,a),\\
\beta_n'(a,k)&= \frac{1-k}{1-k q^{2n}}\left (\frac{k}{a} \right)^n
\alpha_{n}(k,a).
\end{align*}
\end{corollary}

For the present purposes, we may call the pair $(\alpha_{n}'(a,k),
\beta_n'(a,k))$ the \emph{dual} of the pair $(\alpha_{n}(a,k),
\beta_n(a,k))$. Thus, for example, our pair at \eqref{mz02} is the
dual of the WP-Bailey pair of Andrews and Berkovich in \cite{AB02}:
{\allowdisplaybreaks
\begin{align*}
\alpha_n(a,k)&=\begin{cases}
\displaystyle{\frac{(a,q^2\sqrt{a},-q^2\sqrt{a},a^2/k^2;q^2)_{n/2}}{(q^2,
\sqrt{a},-\sqrt{a}, q^2
k^2/a;q^2)_{n/2}}}\left(\frac{k}{a}\right)^n, & n \text{ even},\\
0,& n \text{ odd}.
\end{cases}\\
\beta_n(a,k)&=\frac{(k, \,k \sqrt{q/a},\, -k \sqrt{q/a}, \,a/k;q)_n}
{(\sqrt{a q},\,-\sqrt{a q},\, q
k^2/a,\,q;q)_n}\,\left(\frac{-k}{a}\right)^n.
\end{align*}
}

We had initially thought that the WP-Bailey pair in Lemma \ref{l1}
was the dual of the following pair of Andrews and Berkovich:
\begin{align}\label{ABpr2}
\alpha_n^{(4)}(a,k)&=\frac{1-a q^{2n}}{1-a}\frac{(a,k/aq;q)_n}{(q,
q^2a^2/k;q)_n}\frac{(qa^2/k;q)_{2n}}{(k;q)_{2n}}\left(\frac{k}{a}\right)^n,\\
\beta_n^{(4)}(a,q)&=\frac{(k^2/q a^2;q)_n}{(q;q)_n}. \notag
\end{align}
However, the dual of the pair in Lemma \ref{l1} is the pair
\begin{align}\label{l1pr2}
\alpha_n^{(5)}(a,k)&=\frac{1-a q^{2n}}{1-a}\frac{(a,k q/a;q)_n}{(q,
a^2/k;q)_n}\frac{(a^2/k;q)_{2n}}{(k q;q)_{2n}}\left(\frac{k}{a}\right)^n,\\
\beta_n^{(5)}(a,q)&=\frac{1-k}{1-k q^{2n}}\frac{(q k^2/
a^2;q)_n}{(q;q)_n}, \notag
\end{align}
while the dual of  the pair at \eqref{ABpr2} is the pair
\begin{align}\label{mz01pr2}
\alpha_n^{(6)}(a,k)&=\frac{1-a q^{2n}}{1-a}\frac{(a^2/q
k^2;q)_n}{(q,q)_n}\left(
\frac{k}{a}\right)^n,\\
\beta_n^{(6)}(a,k)&=\frac{( a/k
q,k;q)_n}{(k^2q^2/a,q;q)_n}\frac{(qk^2/a;q)_{2n}}{(a;q)_{2n}}.\notag
\end{align}

Inserting the new WP-Bailey pairs in any of the existing WP-Bailey
chains will lead to  transformations relating basic hypergeometric
series. We believe the following transformations of basic
hypergeometric series to be new.

\begin{corollary}\label{f6c7}
{\allowdisplaybreaks \small{
\begin{multline}\label{f6c71}
_{12} \phi _{11} \left [
\begin{matrix}
k, q\sqrt{k},-q\sqrt{k}, y, z, \displaystyle{ \frac{q a}{k}},
\frac{k}{\sqrt{a}}, - \frac{k}{\sqrt{a}},
k\sqrt{\frac{q}{a}}, -k\sqrt{\frac{q}{a}},\displaystyle{ \frac{kaq^{N+1}}{y z}}, q^{-N}\\
\sqrt{k},-\sqrt{k}, \displaystyle{ \frac{q k}{y},\frac{q k}{z},
 \frac{k^2}{a}, \sqrt{a q},- \sqrt{a q},
q\sqrt{a},-q\sqrt{a},kq^{N+1},\frac{yzq^{-N}}{a}}
\end{matrix}
; q,\,q \right ]\\
=\frac{(k q,k q/y z,a q/y,a q/z;q)_{N}} {(k q/y,k q/z,a q/y z,a
q;q)_{N}} \,  _{5} \phi _{4} \left [
\begin{matrix}
 y,\, z,\,\displaystyle{ \frac{q a^2}{k^2}},\frac{k a q^{N+1}}{y z},q^{-N}\\
\displaystyle{ \frac{q a}{y},\frac{q a}{z},a
q^{N+1},\frac{yzq^{-N}}{k}}
\end{matrix}
; q,\,q \right ].
\end{multline}
} }
\end{corollary}

\begin{proof}
Insert the WP-Bailey pair at \eqref{mz01} into \eqref{wpbteq1} and
set $\rho_1=y$ and $\rho_2=z$.
\end{proof}

Substituting the pair at \eqref{mz01} into \eqref{wpbteq2}  leads to
the following result.
\begin{corollary}
{\allowdisplaybreaks
\begin{multline}
_{7} \phi _{6} \left [
\begin{matrix}
k,\displaystyle{\frac{q a}{k}, \frac{k }{\sqrt{a}},-\frac{k
}{\sqrt{a}},
 k\sqrt{\frac{q}{a}},-k\sqrt{\frac{q}{a}}},q^{-N}\\
\displaystyle{\frac{k^2}{a}},\sqrt{a q},-\sqrt{a q},q
\sqrt{a},-q\sqrt{a},\displaystyle{\frac{k^2q^{-N}}{a^2}}
\end{matrix}
; q,q \right ]=\\
\frac{(qa/k,q a^2/k;q)_N}{( q a,q a^2/k^2;q)_N}\, _{7} \phi _{6}
\left [
\begin{matrix}
\displaystyle{\frac{q a^2}{k^2}}, \sqrt{k}, -\sqrt{k},\sqrt{k q},
-\sqrt{k q},\,\,
\displaystyle{\frac{a^2 q^{N+1}}{k}},q^{-N}\\
\displaystyle{a\sqrt{\frac{q}{k}},-a\sqrt{\frac{q}{k}}, \frac{a
q}{\sqrt{k}},-\frac{a q}{\sqrt{k}},\frac{k q^{-N}}{a}}, a q^{N+1}
\end{matrix}
; q,q \right ].
\end{multline}
}
\end{corollary}

\begin{corollary}
{\allowdisplaybreaks
\begin{multline}\label{id3eq1}
_{14}W_{13}\left( q;y,yq,z,zq, \frac{q^2}{ad},\frac{d
q}{a},q^2,\frac{a q^{N+2}}{yz},\frac{a q^{N+3}}{yz},
q^{-N},q^{1-N};q^2,q^2 \right)-\\
aq\frac{\left( 1 - \frac{d}{a} \right)
    \left( 1 - \frac{q}{ad} \right)
    \left( 1 - q^3 \right)
    \left( 1 - q^{-N} \right) \left( 1 - y \right)
    \left( 1 - \frac{aq^{2 + N}}{yz} \right)
    \left( 1 - z \right) }{\left( 1 - ad \right)
    \left( 1 - q \right)
    \left( 1 - \frac{aq}{d} \right)
    \left( 1 - q^{2 + N} \right)
    \left( 1 - \frac{q^2}{y} \right)
    \left( 1 - \frac{q^2}{z} \right)
    \left( 1 - \frac{yzq^{-N}}{a} \right) }\\
    \times
    \,_{14}W_{13}\left( q^3;yq,yq^2,zq,zq^3,
\frac{q^3}{ad},\frac{d q^2}{a},q^2,\frac{a q^{N+3}}{yz},\frac{a
q^{N+4}}{yz}, q^{1-N},q^{2-N};q^2,q^2 \right)\\
=\frac{\left(q^2,\frac{q^2}{y z},\frac{a q}{y},\frac{a
q}{z};q\right)_{N}} {(\frac{q^2}{y},\frac{q^2}{z},\frac{a q}{y z},a
q;q)_{N}} \,_{10}W_{9}\left( a;y,z,d, \frac{q}{d},-a,\frac{a
q^{N+2}}{yz},q^{-N};q,-a\right).
\end{multline}
} {\allowdisplaybreaks
\begin{multline}\label{cor14eq2}
 _{7} \phi _{6} \left [
\begin{matrix}
q^{5/2},\, -q^{5/2}, \,  y,\,
\displaystyle{\frac{q}{y},\,\frac{q^2}{a d}, \, \frac{d q}{a}},\, q^2\\
\\
q^{1/2},\, -q^{1/2},\,q^2 y,\, \displaystyle{\frac{q^3}{y},\,q a
d,\, \frac{q^2 a}{d}}
\end{matrix}; q^2, a^2 \right ]\\
-\frac{a^2\,\left( 1 - \frac{d}{a} \right) \,
    \left( 1 - \frac{q}{a\,d} \right) \,\left( 1 - q^3 \right) \,
    \left( 1 - \frac{q}{y} \right) \,\left( 1 - y \right) }{\left( 1 -
      a\,d \right) \,\left( 1 - q \right) \,
    \left( 1 - \frac{a\,q}{d} \right) \,
    \left( 1 - \frac{q^2}{y} \right) \,\left( 1 - q\,y \right) }\\
    \times
    \,  _{7} \phi _{6} \left [
\begin{matrix}
q^{7/2},\, -q^{7/2}, \,  y q,\,
\displaystyle{\frac{q^2}{y},\,\frac{q^3}{a d}, \, \frac{d q^2}{a}},\, q^2\\
\\
q^{3/2},\, -q^{3/2},\,q^3 y,\, \displaystyle{\frac{q^4}{y},\,q^2 a
d,\, \frac{q^3 a}{d}}
\end{matrix}; q^2, a^2 \right ]\\
= \frac{(q,q^2;q)_{\infty}(a d y, a d q/y, a y q/d, q^2 a/d y;
q^2)_{\infty}}{(a b,aq/d,y q,q^2/y;q)_{\infty}}.
\end{multline}
}
\end{corollary}

\begin{proof}
Insert the WP-Bailey pair at \eqref{mz03} into \eqref{wpbteq1}, set
$\rho_1=y$ and $\rho_2=z$ and replace $k$ with $q$. For
\eqref{cor14eq2}, set $z=q/y$, let $N \to \infty$ and use
\eqref{whq3f2} to sum the resulting right side.
\end{proof}

\begin{corollary}
{\allowdisplaybreaks
\begin{multline}\label{cor15eq}
_{5} \phi _{4} \left[
\begin{matrix}
\displaystyle{\frac{q^2}{a d}},\frac{d q}{a}, q^{-N},q^{1-N},q^2\\
 \displaystyle{\frac{a q^2}{d}},a d q,\displaystyle{\frac{
q^{2-N}}{a^2}},\displaystyle{\frac{q^{3-N}}{a^2}}
\end{matrix}
; q^2,q^2 \right ]\\ -
 a q \frac{(1-q^{-N})(1-q/ad)(1-d/a)}
 {(1-q^{2-N}/a^2)(1-a d)(1- aq/d)}
\, _{5} \phi _{4} \left[
\begin{matrix}
\displaystyle{\frac{q^3}{a d}},\frac{d q^2}{a}, q^{1-N},q^{2-N},q^2\\
 \displaystyle{\frac{a q^3}{d}},a d q^2,\displaystyle{\frac{
q^{3-N}}{a^2}},\displaystyle{\frac{q^{4-N}}{a^2}}
\end{matrix}
; q^2,q^2 \right ]\\
=\frac{(a,a^2;q)_N}{(q a,a^2/q;q)_N}\,_{10}W_{9}\left(
a;\sqrt{q},-\sqrt{q},d, \frac{q}{d},q,a^2 q^{N},q^{-N};q,-a\right).
\end{multline}
}
\end{corollary}

\begin{proof}
Insert the WP-Bailey pair at \eqref{mz03} into \eqref{wpbteq2},  and
replace $k$ with $q$.
\end{proof}
Remark: The extra $q$-products inserted in each of the series on the
right side of \eqref{id3eq1} and in \eqref{cor15eq} are there to
allow these series to be represented as $_{r+1}\phi_r$ or
$_{r+1}W_r$ series.

\begin{corollary}
{\allowdisplaybreaks
\begin{multline}\label{warcorex1}
\negthickspace
_{12}W_{11}\left(k;\frac{qa}{k},\frac{k}{\sqrt{a}},\frac{-k}{\sqrt{a}},k\sqrt{\frac{q}{a}},-k\sqrt{\frac{q}{a}},
\sqrt{ak}q^N,
-\sqrt{ak}q^N,-q^{-N},q^{-N};q,q^2\right)\\
=\frac{\left( \frac{k}{a}, k^2 q^2,-a, -a q;q^2\right)_N}{\left(
\frac{a}{k}, a^2 q^2,-k q, -k q^2;q^2\right)_N}\left( \frac{a
q}{k}\right)^N\\
\times \, _7\phi_6 \left(
\begin{matrix}
i q\sqrt{a}, -i q \sqrt{a}, \frac{a^2 q}{k}, \sqrt{ak}q^N,
-\sqrt{ak}q^N, -q^{-N}, q^{-N}\\
\\
i \sqrt{a}, -i  \sqrt{a}, a q^{1+N}, -a q^{1+N},
\sqrt{\frac{a}{k}}q^{1-N},-\sqrt{\frac{a}{k}}q^{1-N}
\end{matrix};q,q
\right).
\end{multline}
}
\end{corollary}
\begin{proof}
After replacing $m$ with $k$ and employing some simple
transformations for $q$-products, \eqref{War3} can be rewritten as
\begin{multline}\label{War3n}
\sum_{j=0}^N \frac{1-k q^{2j}}{1-k}\frac{(a k q^{2N},
q^{-2N};q^2)_j}{\left(k^2 q^{2+2N},
\frac{k}{a}q^{2-2N};q^2\right)_j}q^{2j}\beta_j(a,k)\\
= \frac{\left( \frac{k}{a},k^2q^2;q^2\right)_N}
{\left( \frac{a}{k},a^2q^2;q^2\right)_N} \frac{(-a;q)_{2n}}{(-k q;q)_{2n}}\left(\frac{a q}{k}\right)^{N}\\
\times \sum_{j=0}^N\frac{1+a q^{2j}}{1+a}\frac{(a k q^{2N},
q^{-2N};q^2)_j}{\left(a^2 q^{2+2N},
\frac{a}{k}q^{2-2N};q^2\right)_j}\left(\frac{a
q}{k}\right)^{j}\alpha_j(a,k).
 \end{multline}
The result follows, upon inserting the pair from Lemma \ref{l1},
 and rearranging.
\end{proof}

Inserting the new pairs in other WP-Bailey chains will lead to
other, possibly new, transformations of basic hypergeometric series,
but we refrain from further examples here.

{\allowdisplaybreaks
\section{WP-Burge Pairs}
}

In \cite{AB02}, the authors termed a WP-Bailey pair
$(\alpha_{n}(1,k), \beta_{n}(1,k))$ in which $\alpha_{n}(1,k)$ does
not depend on $k$ a \emph{WP-Burge pair.}

One such pair that they derive in \cite[(7.16)]{AB02} is the pair in
the following theorem. For completeness, we give an alternative
proof of this result.
\begin{theorem}
Define {\allowdisplaybreaks
\begin{align}\label{newwp}
\alpha_{n}(1,k)&=
\begin{cases}
1,&n=0,\\
q^{-n/2}+q^{n/2},&n\geq1,
\end{cases}\\
\beta_{n}(1,k)&=
\frac{(k\sqrt{q},k;q)_{n}}{(\sqrt{q},q;q)_{n}}q^{-n/2}. \notag
\end{align}
} Then $(\alpha_{n}(1,k),\beta_{n}(1,k))$ satisfy \eqref{WPpair}
(with $a=1$).
\end{theorem}
\begin{proof}
We begin by recalling Bailey's $\,_{6}\psi_6$ summation formula
\cite{W36}. {\allowdisplaybreaks
\begin{multline*}
\frac{ (aq,aq/bc,aq/bd,aq/be,aq/cd,aq/ce,aq/de,q,q/a;q)_{\infty} } {
(aq/b,aq/c,aq/d,aq/e,q/b,q/c,q/d,q/e,qa^2/bcde;q)_{\infty} }\\ =
\sum_{n=-\infty}^{\infty} \frac{(1-a q^{2n}) (b,c,d,e;q)_{n}}
{(1-a)(aq/b,aq/c,aq/d,aq/e;q)_{n}} \left( \frac{q
a^2}{b c d e}\right)^n\\
=1+\sum_{n=1}^{\infty} \frac{(1-a q^{2n}) (b,c,d,e;q)_{n}}
{(1-a)(aq/b,aq/c,aq/d,aq/e;q)_{n}} \left( \frac{q a^2}{b c d
e}\right)^n\\
+\sum_{n=1}^{\infty}\frac{(1-1/a q^{2n}) (b/a,c/a,d/a,e/a;q)_{n}}
{(1-1/a)(q/b,q/c,q/d,q/e;q)_{n}} \left( \frac{q a^2}{b c d
e}\right)^n,
\end{multline*}
} where the second equality follows from the definition
\[
(z;q)_{-n}= \frac{ (-1)^n q ^{n(n+1)/2} }
                 { z^n(q/z;q)_n}.
\]
Next, set $a=-1$, $b=-c$ and $d=-e$, so that both sums in the
final expression above become equal, to get
\begin{multline*}
\frac{
(-q,q/c^2,q/e^2;q)_{\infty}(q^2/c^2e^2,q^2/c^2e^2,q^2;q^2)_{\infty}
} { (q/c^2e^2;q)_{\infty}(q^2/c^2,q^2/c^2,q^2/e^2,q^2/e^2;q^2)_{\infty} }\\
= 1+\sum_{n=1}^{\infty} \frac{(1+ q^{2n}) (c^2,e^2;q^2)_{n}}
{(q^2/c^2,q^2/e^2;q^2)_{n}} \left( \frac{q }{c^2  e^2}\right)^n,
\end{multline*}
or, upon replacing $c^2$ with $c$ and $e^2$ with $e$,
\begin{multline}\label{ceproducteq}
\frac{ (-q,q/c,q/e;q)_{\infty}(q^2/c e,q^2/c e,q^2;q^2)_{\infty}
} { (q/c e;q)_{\infty}(q^2/c,q^2/c,q^2/e,q^2/e;q^2)_{\infty} }\\
= 1+\sum_{n=1}^{\infty} \frac{(1+ q^{2n}) (c,e;q^2)_{n}}
{(q^2/c,q^2/e;q^2)_{n}} \left( \frac{q }{c  e}\right)^n.
\end{multline}

 Now replace
$q$ with $\sqrt{q}$, set $e=q^{-N}$, $c=k q^{N}$ to get, after
some minor rearrangements, that
\begin{equation*}
1+\sum_{n=1}^{N} \frac{ (k q^N,q^{-N};q)_{n}}
{(q^{1-N}/k,q^{1+N};q)_{n}} \left( \frac{q }{k}\right)^n(q^{-n/2}+
q^{n/2})=\frac{ (k \sqrt{q},q;q)_{N}} {(k,\sqrt{q};q)_{N}}\,q^{-N/2}
\end{equation*}
The result now follows from \eqref{WPpair}, after some simple
manipulations.
\end{proof}

One reason this WP-Bailey pair is somewhat interesting is that the
special case $k=0$ of \eqref{newwp} gives Slater's standard Bailey
pair \textbf{F3} from \cite{S51}. Recall that a pair of sequences
$(\alpha_n, \beta_n)$ is termed a \emph{Bailey pair relative to
$x/q$}, if they satisfy
\begin{equation}\label{bpeq}
\beta_n = \sum_{r=0}^{n} \frac{\alpha_r}{(q;q)_{n-r}(x;q)_{n+r}},
\end{equation}

We  also recall the following result of Bailey, a particular case of
the ``Bailey Transform", from his 1949 paper \cite{B49}.

\begin{theorem}\label{bt}
Subject to suitable convergence conditions, if the sequences
$\{\alpha_n\}$ and $\{\beta_n\}$ satisfy \eqref{bpeq}, then
\begin{equation}\label{Baileyeq} \sum_{n=0}^{\infty}
(y,z;q)_{n}\left ( \frac{x}{yz}\right )^{n} \beta_n = \frac{(x/y,x
/z;q)_{\infty}}{ (x , x/yz;q)_{\infty}} \sum_{n=0}^{\infty}
\frac{(y,z;q)_{n}}{(x/y,x/z;q)_n}\left ( \frac{x}{yz}\right )^{n}
\alpha_n.
\end{equation}
\end{theorem}

Thus the
 pair at \eqref{newwp} ``lifts" this standard Bailey pair to a
WP-Bailey pair and lifts all the series--product identities
following from \textbf{F3} to more general series--product
identities containing the free parameter $k$.

More generally, let  $(\alpha_n(a,k), \beta_n(a,k))$ be a WP-Bailey
pair in which $\alpha_n(a,k)$ is independent of $k$. Suppose further
that this pair is a ``lift" of a standard Bailey pair, in the sense
that setting $k=0$ in the WP-Bailey pair recovers the Bailey pair.
By comparing \eqref{Baileyeq} and \eqref{wpbteq1b} (first setting
$y=\rho_1$, $z=\rho_2$ and $x=aq$ in \eqref{Baileyeq}) we see that
the two infinite series containing $\alpha_n=\alpha_n(a,k)$ are
identical. This means that if particular choices of $y$ and $z$ in
\eqref{Baileyeq} lead to an identity of the Rogers-Ramanujan type,
then the same choices for $\rho_1$ and $\rho_2$ in \eqref{wpbteq1b}
will lead to a more general series--product identity containing an
extra free parameter, namely $k$, and this more general identity
will revert back to the original identity of Rogers-Ramanujan type,
upon setting $k=0$.

The substitution of this pair at \eqref{newwp} into
\eqref{wpbteq1} (first setting $\rho_1=y$, $\rho_2=z$ in
\eqref{wpbteq1} as above, and then setting $a=1$) and
\eqref{wpbteq1b} (in addition, replacing $q$ with $q^2$ and
specializing $y$ and $z$) leads to the following corollary.
\begin{corollary}
\begin{multline}\label{q1/2pair}
 _{8} \phi _{7} \left [
\begin{matrix}
k,q\sqrt{k}, -q\sqrt{k},y,z,
k\sqrt{q}, kq^{1+N}/yz,q^{-N}\\
\sqrt{ k},-\sqrt{k}, q k/y,q k/z,\sqrt{q},k q^{1+N}, y z q^{-N}
\end{matrix}
; q,\sqrt{q} \right ]\\=
\frac{(qk,qk/yz,q/y,q/z;q)_{N}} {(qk/y,qk/z,q,q/yz;q)_{N}}\\
\times \left( 1+\sum_{n=0}^{N}\frac{(1+q^{n})(y,z,
kq^{1+N}/yz,q^{-N};q)_{n}}{(q/y,q/z, q^{1+N},y z
q^{-N}/k;q)_n}\left (\frac{\sqrt{q} }{k}\right)^{n} \right).
\end{multline}
\begin{equation}\label{q1/2pairpr}
\sum_{n=0}^{\infty} \frac{ (1-k q^{4n})(k;q)_{2n}q^{2n^2-n} }{
(1-k)(q;q)_{2n}}= \frac{(kq^2;q^2)_{\infty}}{(q;q^2)_{\infty} }.
\end{equation}
\begin{equation}\label{q1/2pairpr2}
\sum_{n=0}^{\infty} \frac{ (1-k q^{2n})(k;q)_{n}(-1)^nq^{n(n-1)/2}
}{ (1-k)(q;q)_{n}}= 0.
\end{equation}
\begin{equation}\label{q1/2pairpr3}
\sum_{n=0}^{\infty} \frac{ (1-k
q^{4n})(-q;q^2)_n(k;q)_{2n}q^{n^2-n} }{
(1-k)(-kq;q^2)_n(q;q)_{2n}}=
\frac{(kq^2,-1;q^2)_{\infty}}{(-kq,q;q^2)_{\infty} }.
\end{equation}
\end{corollary}

The last three identities  also follow as special cases of
Jackson's $_8\phi_7$ summation formula. However, they do
illustrate how a lift of a standard Bailey pair leads to
generalizations of identities arising from this standard pair (the
identities given by setting $k=0$ in the corollary above).

We will  investigate this phenomenon of WP-Bailey pairs that are
lifts of standard Bailey pairs further in a subsequent paper.

 \allowdisplaybreaks{

}
\end{document}